\documentclass[11pt]{article}

\usepackage{amsmath, amssymb, amsthm}   % For math symbols and environments
\usepackage{mathrsfs}                  % For script fonts, if needed
\usepackage{graphicx}                  % For figures, if used later
\usepackage{enumitem}                  % For customizable lists
\usepackage{hyperref}                  % For hyperlinks
\usepackage{xcolor,tikz}                    % For colored text (optional)
\usepackage[ruled,vlined,linesnumbered]{algorithm2e}
\usepackage{graphicx}

\usepackage{lmodern}                  % Better font rendering
\usepackage[T1]{fontenc}
\usepackage[utf8]{inputenc}

\newtheorem{theorem}{Theorem}[section]

\newtheorem{lemma}[theorem]{Lemma}

\title{Universal Cycles on Affine Lines}
\author{Ming-Hsuan Kang, Shin-Hsun Chou}
\date{April 2026}

\def \F{\mathbb{F}} 
\def \AG{\mathrm{AG}}

\def \GL{\mathrm{GL}}

\begin{document}

\maketitle

\begin{abstract}
A universal cycle is a cyclic sequence in which each object of a combinatorial family appears exactly once as a contiguous window. 
While such cycles are well understood for many discrete structures and linear subspaces, the case of affine lines 
presents additional difficulties arising from parallelism.

We prove that universal cycles exist for affine lines in $\mathrm{AG}(n,q)$ for all $n \ge 2$ and all prime powers $q$. 
Our construction embeds the problem into $\mathrm{PG}(n,q)$, using points at infinity to encode directions, and proceeds via a decomposition into pairwise and triple configurations combined with a recursive lifting and gluing argument.

We further interpret the construction in the Grassmannian $G_q(2,n+1)$, where affine lines correspond to the outer shell of $2$-subspaces, thereby extending known constructions for Grassmannians. 
A Python implementation is provided as supplementary material.
\end{abstract}

\section{Introduction}

Universal cycles, introduced by Chung, Diaconis, and Graham~\cite{ChungDiaconisGraham},
provide compact cyclic encodings of large combinatorial families.
A universal cycle for a family $\mathcal{F}$ of discrete objects is a cyclic
sequence in which every object of $\mathcal{F}$ appears exactly once as a
contiguous subsegment (or “window”).  
The classic example is the de~Bruijn cycle, which efficiently encodes all
$q$-ary words of length $n$.
Since their introduction, universal cycles have been constructed for many
combinatorial families, including subsets, permutations, partitions,
and designs; see, for example,
\cite{HurlbertJohnsonZahl2009,KnuthVol4A,Johnson}.
These constructions are closely related to Gray codes and combinatorial
generation problems, where one seeks cyclic listings with prescribed local transition rules; see \cite{SavageSurvey} for a general overview.

A major development in the geometric setting was initiated by
Jackson~\cite{Jackson2009}, who proved the first general existence theorem for
universal cycles on $2$-dimensional subspaces of the Grassmannian $G_q(2,m)$.
His recursive construction proceeds by partitioning $G_q(2,m)$ into “special’’
and “regular’’ subspaces, distinguished by whether they contain a fixed
coordinate vector.
In contrast to recursive approaches, algebraic constructions such as
\cite{ChiKangHsieh2025} provides explicit cycles using the multiplicative
structure of finite fields, in a manner analogous to classical de~Bruijn
sequences.

\medskip

While most existing results focus on linear subspaces or combinatorial objects with global symmetry, the case of affine lines introduces additional challenges arising from parallelism and the absence of a canonical origin.
In particular, affine lines naturally decompose into parallel classes,
and this structure is not directly compatible with the standard universal-cycle
frameworks.

In this paper, we investigate universal cycles for the family of
\emph{affine lines} in finite affine space.
Let $\mathrm{AG}(n,q)\cong\mathbb{F}_q^{\,n}$ denote the $n$-dimensional affine
space over $\mathbb{F}_q$.
An affine line $\ell\subset\mathrm{AG}(n,q)$ is uniquely determined by any
ordered pair of distinct points $(u,v)$ lying on it.
Thus a cyclic vertex sequence
\[
   \mathcal{C}=(p_i)_{i\in\mathbb{Z}/N\mathbb{Z}}
\]
induces a cyclic sequence of affine lines via windows $(p_i,p_{i+1})$.
The universal-cycle problem asks whether such a sequence can represent
every affine line of $\mathrm{AG}(n,q)$ exactly once.

\subsection{The Affine Obstruction and the Projective Extension}

A fundamental obstruction arises if one attempts to construct such a cycle
using only affine points.
Consider $\mathrm{AG}(n,2)$.
Since each line contains exactly two points, a universal cycle for lines would
be a cyclic ordering of $\mathbb{F}_2^{\,n}$ in which every unordered pair of
distinct points appear exactly once as an adjacent pair.
Equivalently, this corresponds to an Eulerian circuit in the complete graph $K_{2^n}$.
Because all vertices of $K_{2^n}$ have odd degree, no Eulerian circuit exists.
Thus, purely affine universal cycles are impossible for $q=2$.

To overcome this obstruction, we extend the vertex set to the projective
completion $\mathrm{PG}(n,q)$ by adjoining the hyperplane at infinity
$H_\infty$.
Every affine line $\ell$ has a unique point at infinity $[\ell]\in H_\infty$
representing its direction.
Allowing these points at infinity to appear in the cycle provides the
flexibility needed to capture the parallelism structure of affine lines.

\paragraph{Example.}
Let $\mathbb{F}_2^{\,2}=\operatorname{span}\{u_1,u_2\}$ and define
\[
   \ell_1=\mathbb{F}_2 u_1,\qquad
   \ell_2=\mathbb{F}_2 u_2,\qquad
   \ell_3=\mathbb{F}_2(u_1+u_2),
\]
the three directions of $\mathrm{AG}(2,2)$.
Writing $[\ell_i]$ for their points at infinity, the cyclic sequence
\[
   0 \;\to\; [\ell_1] \;\to\; u_1+u_2 \;\to\; [\ell_3]
   \;\to\; u_1 \;\to\; [\ell_2] \;\to\; 0
\]
represents every affine line of $\mathrm{AG}(2,2)$ exactly once.

\subsection{Main Result}

Our main theorem establishes the existence of universal cycles for affine
lines in all finite dimensions.

\begin{theorem}[Main Theorem]
Let $n\ge 2$ and $q$ be any prime power.
There exists a universal cycle for the affine lines of $\mathrm{AG}(n,q)$,
that is, a cyclic sequence
\[
   \mathcal{C}=(p_i)_{i\in\mathbb{Z}/N\mathbb{Z}} \subset \mathrm{PG}(n,q)
\]
such that:
\begin{enumerate}
   \item each window $(p_i,p_{i+1})$ represents a well-defined affine line of
         $\mathrm{AG}(n,q)$;
   \item every affine line of $\mathrm{AG}(n,q)$ appears exactly once as such
         a window.
\end{enumerate}
The vertex set may include both affine points and points at infinity, as
required by the construction.
\end{theorem}

\medskip
Our construction also admits a Grassmannian interpretation. 
Via a standard homogenization map, affine lines in $\mathrm{AG}(n,q)$ correspond precisely to the outer shell of the Grassmannian $G_q(2,n+1)$. 
In this sense, our universal cycles extend the recursive framework originating in Jackson’s work for full Grassmannians.

Our main contribution is to establish the existence of universal cycles for affine lines in all dimensions. 
This is achieved through a projective extension framework, together with a structural decomposition into pairwise and triple constructions, unified via a recursive lifting mechanism.

To facilitate verification and reproducibility, we provide a Python implementation of the universal cycle construction as supplementary material on arXiv. 
Portions of the exposition and implementation were developed with the assistance of the AI tool ChatGPT.

\section{Preliminaries}

\subsection{Segments, Cycles, and Universal Cycles}

A finite sequence of points
\[
   (p_1,\dots,p_{m+1})\subset\mathrm{PG}(n,q)
\]
is called a \emph{double-window segment} if:
\begin{enumerate}
   \item each window $(p_i,p_{i+1})$ determines an affine line, and
   \item these windows represent pairwise distinct affine lines.
\end{enumerate}
We require $p_1\ne p_{m+1}$ so that the segment has two distinct endpoints.
If instead $p_{m+1}=p_1$, the sequence is a \emph{double-window cycle}.

Throughout, all uses of “segment’’ and “cycle’’ refer to double-window objects.

\medskip
\paragraph{Transversality.}
Two segments or cycles are \emph{transversal} if the sets of affine lines they
represent are disjoint.  
Equivalently, all windows appearing in the two objects represent distinct lines.

\medskip
\paragraph{Universal cycle.}
Given a set $\mathcal{F}$ of affine lines, a double-window cycle is
\emph{universal for $\mathcal{F}$} if its windows represent exactly the lines
in $\mathcal{F}$.

\subsection{Gluing Lemma}

We record the two basic operations for assembling cycles from smaller pieces.
Transversality is always assumed.

\medskip
\noindent\textbf{Gluing cycles.}
If several transversal cycles share a common vertex, they may be spliced at
that vertex to form a single larger cycle.

\medskip
\noindent\textbf{Gluing segments.}
Let $\mathcal{S}$ be a family of transversal segments, each with two distinct
endpoints.  
For a vertex $v$, let $\mu(v)$ be the number of segments in $\mathcal{S}$ having
$v$ as an endpoint.  
If $\mu(v)$ is even for every vertex, then by reversing some segments if
needed, all segments can be concatenated into a single double-window cycle.

\medskip
These facts follow immediately from the standard parity condition for joining
finite paths, and we omit the proof.

\subsection{Parametrization of Direction Fibers}

For a $1$-dimensional subspace $\ell\subset\mathbb{F}_q^n$, its point at
infinity is denoted $[\ell]$.  
The \emph{direction fiber} of $\ell$ is
\[
   \mathcal{L}(\ell)=\{\,x+\ell : x\in\mathbb{F}_q^n\,\},
\]
the set of all affine lines parallel to $\ell$.

Fix a decomposition
\[
   \mathbb{F}_q^n=\ell\oplus W.
\]
Each line $L\in\mathcal{L}(\ell)$ meets $W$ in a unique point, giving a natural
parametrization
\[
   \Phi_\ell(w)=w+\ell \qquad (w\in W).
\]
Thus each $L\in\mathcal{L}(\ell)$ is encoded by the 2-term segment
$(w,[\ell])$.

\subsection{Strategy of Construction}

With the parametrization of fibers in place, our construction proceeds in four logical steps:

\begin{enumerate}
   \item \textbf{Pairwise Construction:} We first show that any two direction fibers admit a universal cycle. This serves as the primary building block for the majority of directions.
   \item \textbf{Recursive Lifting:} We establish a dimension-reduction tool (the Recursive Lifting Lemma), proving that a universal cycle on a subspace can be naturally extended to the full affine space.
   \item \textbf{Triple Construction:} Utilizing the lifting lemma, we show that any three coplanar fibers admit a universal cycle. This construction is crucial for handling the parity constraint when the total number of directions is odd.
   \item \textbf{Assembly:} We partition the set of all directions based on the parity of $[n]_q$. Using the Gluing Lemma, we assemble the pairwise and triple cycles into a single universal cycle covering all affine lines in $\mathrm{AG}(n,q)$.
\end{enumerate}

\section{Universal Cycles for Unions of Direction Fibers}

With the parametrization of direction fibers in place, we now analyze how
universal cycles may be constructed on unions of multiple fibers.  
We begin with the fundamental two–fiber case, then treat the planar three–fiber
configuration in an even characteristic, and finally establish the five–fiber
construction for the odd case in higher dimensions.  
These results form the essential ingredients for assembling universal cycles on
all affine lines of $\mathrm{AG}(n,q)$.

\subsection{Two Fibers}

We first show that whenever two directions admit a common transversal
hyperplane, their fibers together form a universal cycle.

\begin{theorem}\label{jointtwofiber}
Let $\ell_1,\ell_2 \subset \mathbb{F}_q^{\,n}$ be distinct $1$-dimensional
subspaces.  
Suppose a codimension-one subspace $W \subset \mathbb{F}_q^{\,n}$ satisfies
\[
   \mathbb{F}_q^{\,n} = \ell_1 \oplus W = \ell_2 \oplus W.
\]
Then the union
\[
   \mathcal{L}(\ell_1)\,\sqcup\,\mathcal{L}(\ell_2)
\]
admits a universal cycle whose affine vertices are precisely the points of $W$.
In particular, the cycle contains the vertex $0$.
\end{theorem}

\begin{proof}
Each fiber is uniformly parametrized by
\[
   \Phi_{\ell_i}(w)=w+\ell_i, \qquad w\in W,\ i=1,2.
\]
Thus, any universal cycle on the two fibers becomes a lifted walk alternating
between vertices of $W$ and the points $[\ell_1],[\ell_2]$.

\subsubsection*{Case $q$ even}
Since $|W|=q^{\,n-1}$ is even, write $W=\{w_1,\dots,w_{2N}\}$ and consider the
cycle
\[
   w_1 \to [\ell_1] \to w_2 \to [\ell_2] \to w_3 \to [\ell_1] 
        \to \cdots \to w_{2N} \to [\ell_2].
\]
Each $w_k$ appears exactly once with each direction, so the cycle covers all
lines of both fibers.

\subsubsection*{Case $q$ odd}
Now $|W|=2N+1$ is odd.  
Because $\ell_1\oplus\ell_2$ meets $W$ in a unique nonzero direction, write
\[
   W\cap(\ell_1\oplus\ell_2)=\mathbb{F}_q w^{*},\qquad w^{*}\neq 0,
\]
and enumerate $W\setminus\{w^{*}\}=\{0=w_1,w_2,\dots,w_{2N}\}$.
Applying the even-$|W|$ construction gives a cycle covering all lines except
\[
   \Phi_{\ell_1}(w^{*}),\qquad
   \Phi_{\ell_2}(w^{*}).
\]

Write $w^{*}=a u_1 + b u_2$ with $a,b\neq 0$.  
Replace the initial segment $0\to[\ell_1]$ by
\[
   0 \to a u_1 \to w^{*} \to [\ell_1].
\]
The windows introduced are exactly the missing lines, and no other changes
occur. Hence, the completed cycle covers all lines in both fibers and contains $0$.
\end{proof}

\subsection{The Recursive Lifting Lemma}

To simplify the construction for arbitrary dimensions, we introduce a method to extend universal cycles from a subspace to the entire affine space.
Before presenting the formal proof, we provide the geometric intuition behind this reduction.

\subsubsection*{Geometric Intuition: The Stack of Layers}
We may visualize the affine space $\mathbb{F}_q^{\,n}$ as a ``stack of layers,'' where each layer corresponds to a coset of the subspace $U$.
Since the set of directions $\mathcal{D}$ is contained entirely within $U$, any affine line with a direction in $\mathcal{D}$ is confined to a single layer; it never traverses between layers.
Consequently, the problem initially decouples into independent sub-problems on each layer.

The key to the construction is that while the affine points are distinct across layers, the \textit{points at infinity} (the directions) are identical for all layers in the projective completion.
These shared vertices at infinity act as ``pivots'' or ``hinges.''
Our strategy is to construct a cycle on the base layer, replicate it identically on all other layers, and then utilize these shared pivots to stitch the disjoint layer-cycles into a single connected structure covering the entire space.

\begin{lemma}\label{recursivelifting}
Let $U$ be a linear subspace of $\mathbb{F}_q^{\,n}$ with $\dim(U) < n$.
Let $\mathcal{D}$ be a set of directions such that every direction in $\mathcal{D}$ lies in $U$ (i.e., its direction vectors lie in $U$).
Suppose the union of fibers is restricted to the affine subspace $U$
\[
    \bigsqcup_{\ell \in \mathcal{D}} \mathcal{L}_U(\ell)
\]
admits a universal cycle $\mathcal{C}_U$, where $\mathcal{L}_U(\ell)$ denotes the set of affine lines in $U$ with direction $\ell$.
Then the corresponding union of fibers in the full space $\mathbb{F}_q^{\,n}$ admits a universal cycle $\mathcal{C}$.
Moreover, if $\mathcal{C}_U$ contains the origin $0$, then $\mathcal{C}$ also contains $0$.
\end{lemma}

\begin{proof}
Let $k = \dim(U)$. We decompose the affine space $\mathbb{F}_q^{\,n}$ into $q^{n-k}$ parallel affine subspaces (cosets), say $U_1, \dots, U_{q^{n-k}}$, where $U_1 = U$.
Since every direction $\ell \in \mathcal{D}$ lies in $U$, any affine line with direction $\ell$ is entirely contained within exactly one of these cosets.
Thus, the fiber of $\ell$ in the full space partitions as:
\[
    \mathcal{L}(\ell) = \bigsqcup_{j=1}^{q^{n-k}} \mathcal{L}_{U_j}(\ell).
\]
We construct the cycle $\mathcal{C}$ as follows:
\begin{enumerate}
    \item For each coset $U_j$, we construct a cycle $\mathcal{C}_j$ isomorphic to $\mathcal{C}_U$ via translation.
    \item The vertices of $\mathcal{C}_j$ consist of affine points in $U_j$ and the points at infinity $\{[\ell] : \ell \in \mathcal{D}\}$.
    \item While the affine vertices are disjoint for distinct $j$, the points at infinity are shared across all cycles $\mathcal{C}_j$.
\end{enumerate}
Since the collection of cycles $\{\mathcal{C}_j\}$ shares common vertices (the points at infinity), by the gluing lemma, their union forms a connected Eulerian graph.
Merging these cycles yields a single universal cycle $\mathcal{C}$ covering all lines in $\bigsqcup_{\ell \in \mathcal{D}} \mathcal{L}(\ell)$.
Since $U_1=U$ contains $0$, and $\mathcal{C}_1 = \mathcal{C}_U$ contains $0$, the final merged cycle contains $0$.
\end{proof}

\subsection{Three Fibers}

While Theorem~\ref{jointtwofiber} efficiently handles pairs of fibers, the case where the total number of directions is odd requires a construction for a triplet of directions.
Using the Recursive Lifting Lemma, we can reduce this problem to the plane.

\begin{theorem}\label{jointhreefiber}
Let $\ell_1, \ell_2, \ell_3$ be three distinct \textbf{coplanar} directions in $\mathbb{F}_q^{\,n}$.
Then the union of fibers
\[
   \bigsqcup_{i=1}^3 \mathcal{L}(\ell_i)
\]
admits a universal cycle containing the vertex $0$.
\end{theorem}

\begin{proof}
Let $V' = \mathrm{span}(\ell_1, \ell_2, \ell_3)$. Since the directions are coplanar, $\dim(V') = 2$.
By the Recursive Lifting Lemma (Lemma~\ref{recursivelifting}), it suffices to construct the universal cycle within the affine plane $V' \cong \mathbb{F}_q^{\,2}$.

\paragraph{Normalization of Coordinates.}
Since the group $PGL(2,q)$ acts 3-transitively on the set of directions in $\mathbb{F}_q^{\,2}$, there exists a linear automorphism $\phi: V' \to \mathbb{F}_q^{\,2}$ that maps the three given directions $\ell_1, \ell_2, \ell_3$ to the standard reference directions:
\[
    \ell'_1 = \langle (0,1) \rangle, \quad
    \ell'_2 = \langle (1,0) \rangle, \quad
    \ell'_3 = \langle (1,1) \rangle.
\]
Thus, without loss of generality, we construct the universal cycle $\mathcal{C}'$ for these standard directions. The universal cycle for the original directions is then obtained by applying the inverse transformation $\phi^{-1}$ to every affine vertex in $\mathcal{C}'$, while the direction vertices map back to $[\ell_1], [\ell_2], [\ell_3]$.
The construction of $\mathcal{C}'$ depends on the parity of $q$.

\subsubsection*{Case $q$ is even}
Since the characteristic is 2, we partition the field elements $\mathbb{F}_q$ into $q/2$ pairs of the form $\{u, u+1\}$.
For each pair, we construct a block cycle $\mathcal{C}_u$ covering the 6 lines associated with the coordinates $\{u, u+1\}$.
The cycle is explicitly defined by the sequence:
\[
    \mathcal{C}_u: \quad
    (u, u+1) \to [\ell'_1] \to (u+1, 0) \to [\ell'_3] \to (0, u) \to [\ell'_2] \to (u, u+1).
\]
These cycles are disjoint in affine lines and share the direction vertices. They can be glued into a single universal cycle containing $0$ (by choosing the pair containing $0$).

\subsubsection*{Case $q$ is odd}
We use the "Kernel and Pairs" strategy. Since $q$ is odd, the characteristic is not 2, which implies that the elements $0, 1, 2$ are distinct in $\mathbb{F}_q$.

\paragraph{1. The Kernel Cycle.}
We construct a specific cycle $\mathcal{C}_{\mathrm{ker}}$ covering the 9 lines associated with the indices in $K=\{0, 1, 2\}$.
We use the following explicit sequence:
\[
    \mathcal{C}_{\mathrm{ker}}: \quad
    (0,1) \to (0,2) \to [\ell'_3] \to (2,0) \to (0,0) \to (1,1) \to [\ell'_1] \to (2,2) \to [\ell'_2] \to (0,1).
\]
This sequence covers the required kernel lines and contains the origin $(0,0)$.

\paragraph{2. Generic Block Cycles.}
The set of remaining elements $R = \mathbb{F}_q \setminus K$ has size $q-3$, which is even. We partition $R$ into disjoint pairs. For each pair $\{u, v\} \subset R$, we construct a block cycle:
\[
    \mathcal{C}_{u,v}: \quad
    (u, u) \to [\ell'_1] \to (v, 0) \to [\ell'_3] \to (u+v, v) \to [\ell'_2] \to (u, u).
\]
All block cycles and the kernel cycle share the common direction vertices $[\ell'_1], [\ell'_2], [\ell'_3]$ and can be glued at the origin.
\end{proof}

\subsection{Proof of the Main Theorem}

Let $[n]_q = 1+q+\cdots+q^{\,n-1}$ denote the total number of directions (points at infinity) in $\mathrm{AG}(n,q)$.
To prove the Main Theorem, we construct a universal cycle covering the union of fibers over all directions $\mathcal{D}$ in $\mathrm{AG}(n,q)$.
The strategy relies on partitioning $\mathcal{D}$ based on the parity of $[n]_q$.

\begin{proof}[Proof]
We distinguish two cases regarding the size of the set of directions $\mathcal{D}$.

\medskip
\noindent
\textbf{Case 1: The number of directions $[n]_q$ is even.}
(This occurs, for instance, when $q$ is odd and $n$ is even).
Since $|\mathcal{D}|$ is even, we can partition the set of all directions into $k = [n]_q / 2$ disjoint pairs:
\[
    \mathcal{D} = \{P_1, P_2, \dots, P_k\}, \quad \text{where } P_j = \{\ell_{a,j}, \ell_{b,j}\}.
\]
For each pair $P_j$, we apply Theorem~\ref{jointtwofiber}. This theorem guarantees the existence of a universal cycle $\mathcal{C}_j$ covering the fibers $\mathcal{L}(\ell_{a,j}) \cup \mathcal{L}(\ell_{b,j})$ such that $\mathcal{C}_j$ contains the origin $\mathbf{0}$.
The collection of cycles $\{\mathcal{C}_1, \dots, \mathcal{C}_k\}$ represents disjoint sets of affine lines (since fibers are disjoint). However, they all share the common affine vertex $\mathbf{0}$.
By the Gluing Lemma, the union of these cycles forms a connected Eulerian graph, which can be traversed as a single universal cycle.

\medskip
\noindent
\textbf{Case 2: The number of directions $[n]_q$ is odd.}
Since the total set of directions $\mathcal{D}$ has odd cardinality, we cannot partition it entirely into pairs. We handle the parity by isolating a triplet.

We select a triplet $T = \{\ell_1, \ell_2, \ell_3\} \subset \mathcal{D}$ such that $\ell_1, \ell_2, \ell_3$ lie in a common 2-dimensional subspace.
(Such a triplet always exists because any 2-dimensional subspace of $\mathbb{F}_q^{\,n}$ contains exactly $q+1$ directions, and since $q \ge 2$, we have $q+1 \ge 3$).

We then partition the remaining directions into disjoint pairs:
\[
    \mathcal{D} = T \cup \{P_1, \dots, P_m\},
\]
where $m = ([n]_q - 3)/2$.

\begin{enumerate}
    \item \textbf{The Triplet:} Apply Theorem~\ref{jointhreefiber} to the \textbf{coplanar} triplet $T$. This yields a universal cycle $\mathcal{C}_{\mathrm{trip}}$ covering $\bigcup_{\ell \in T} \mathcal{L}(\ell)$ that contains the origin $0$.
    \item \textbf{The Pairs:} Apply Theorem~\ref{jointtwofiber} to each of the remaining pairs $P_j$. This yields cycles $\mathcal{C}_j$ containing $0$.
\end{enumerate}

The cycle $\mathcal{C}_{\mathrm{trip}}$ and all pairwise cycles $\mathcal{C}_j$ share the common vertex $0$.
Since the edge sets (affine lines) covered by these cycles are disjoint and their union covers all directions in $\mathcal{D}$, applying the Gluing Lemma at $0$ produces a single universal cycle for all affine lines in $\mathrm{AG}(n,q)$.
\medskip
In both cases, for any dimension $n \ge 2$ and any prime power $q$, a universal cycle exists.
\end{proof}

\section{Grassmannian Interpretation and Connections}

The affine-line constructions developed above admit a natural reformulation in
the Grassmannian of $2$-subspaces.  This viewpoint shows that our main theorem
provides, recursively, a universal cycle on $G_q(2,n)$ for all $n$.

\subsection{Affine lines as Grassmannian \(2\)-subspaces}

Let \(G_q(k,m)\) be the Grassmannian of \(k\)-subspaces of \(\F_q^{\,m}\).  
For the coordinate hyperplane
\[
   H_{m-1} := \{ x_m = 0 \} \subset \F_q^{\,m},
\]
define the \emph{outer shell}
\[
   G_q(2,m)^{\circ}
   := \{\, U \in G_q(2,m) : U \not\subset H_{m-1} \,\}.
\]
Thus
\[
   G_q(2,m) \;=\; G_q(2,m-1)\;\sqcup\; G_q(2,m)^{\circ}.
\]

Identify \(\AG(m-1,q)\) with the affine chart
\[
   \{ (x,1) : x \in \F_q^{\,m-1} \} \subset \F_q^{\,m}.
\]
If \(L = w+\ell\) is an affine line in this chart, its homogeneous lifts
\((w,1)\) and \((v,0)\) (for any nonzero \(v\in\ell\)) span a unique
outer-shell subspace:
\[
   \tau(L) := \operatorname{span}\{ (w,1),(v,0) \} \in G_q(2,m)^{\circ}.
\]
Thus
\[
   \tau : \{\text{affine lines in }\AG(m-1,q)\}
          \xrightarrow{\;\cong\;} G_q(2,m)^{\circ}
\]
is a bijection.  
Consequently, any universal cycle on affine lines induces a universal cycle on
the outer shell \(G_q(2,m)^{\circ}\), with windows determined by the lifted
vertex vectors.

\subsection{From outer shells to full Grassmannian cycles}

We now explain how the affine-line construction yields a \emph{nested} family
of universal cycles on the Grassmannians \(G_q(2,n)\).

\medskip
\noindent
\textbf{Base case: \(G_q(2,3)\).}
Via the identification \(\F_{q^3}\cong\F_q^{3}\), a generator
\(\alpha\in\F_{q^3}^\times\) produces the classical Singer cycle
\[
   1 \to \alpha \to \alpha^2 \to \cdots \to \alpha^{q^2+q}\to 1,
\]
whose windows enumerate all elements of \(G_q(2,3)\).  
After choosing coordinates so that \(1\mapsto e_1=(1,0,0)\), we obtain a
universal cycle \(U_3\) on \(G_q(2,3)\) containing \(e_1\).

\medskip
\noindent
\textbf{Equivariance.}  
Because \(\GL_m(\F_q)\) acts transitively on nonzero vectors and preserves
Grassmannians, any universal cycle containing a nonzero vector \(v\) may be
post-composed with a linear map sending \(v\) to \(e_1\).  
Thus we may assume throughout that every cycle we construct contains the
distinguished vertex \(e_1\).

\medskip
\noindent
\textbf{Inductive extension.}
Suppose we have a universal cycle \(U_m\) on \(G_q(2,m)\) containing \(e_1\).
Applying the main affine-line theorem to \(\AG(m,q)\), and transporting the
construction via the bijection \(\tau\), we obtain a universal cycle
\(U_{m+1}^{\circ}\) on the outer shell \(G_q(2,m+1)^{\circ}\), which may again
be normalized to contain \(e_1\).

The crucial feature of our affine-line construction is that the outer-shell
cycle \(U_{m+1}^{\circ}\) is \emph{completely disjoint} from \(U_m\) and is
attached only at the shared vertex \(e_1\).  
Thus, by the gluing lemma, splicing at \(e_1\) produces a universal cycle
\[
   U_{m+1} \quad\text{on}\quad G_q(2,m+1)
\]
such that the original cycle \(U_m\) remains embedded \emph{verbatim} as an
unchanged subcycle of \(U_{m+1}\).

\medskip
Iterating this procedure yields the following structural theorem.

\begin{theorem}\label{thm:nested-Grassmann-cycles}
There exists a family of universal cycles
\[
   U_n \quad\text{on}\quad G_q(2,n), \qquad n\ge 3,
\]
such that:
\begin{enumerate}
   \item Each \(U_n\) contains the distinguished vertex \(e_1\).
   \item For every \(3 \le m \le n\), the entire cycle \(U_m\) appears
         unchanged as a contiguous subcycle of \(U_n\).
\end{enumerate}
\end{theorem}

\newpage

\section*{Appendix: Algorithmic Description}
In this appendix, we present an algorithmic formulation of the construction of universal cycles for affine lines in $\mathrm{AG}(n,q)$, emphasizing the triple and pairwise decompositions together with the lifting process across dimensions.

\SetKwComment{tcc}{\(\triangleright\) }{}
\SetKwComment{tcp}{\(\triangleright\) }{}

\begin{algorithm}
\caption{Construction of Universal Cycle for Affine Lines in $\mathrm{AG}(n,q)$}
\label{alg:main}
\SetKwInOut{Input}{Input}
\SetKwInOut{Output}{Output}

\Input{Dimension $n \ge 2$, Prime power $q$}
\Output{A sequence $S$ representing the universal cycle}

\BlankLine
Initialize empty sequence $S$ containing origin $\mathbf{0}$\;
Let $\mathcal{D}$ be the set of all directions in $\mathrm{AG}(n,q)$\;
$N \leftarrow |\mathcal{D}|$\;

\BlankLine
\tcc{Case: Odd number of directions (Triple Construction)}
\If{$N$ is odd}{
    Select a triplet of directions $T = \{\ell_1, \ell_2, \ell_3\} \subset \mathcal{D}$\;
    $\mathcal{D} \leftarrow \mathcal{D} \setminus T$\;
    \tcp{Construct base cycle in 2D}
    $C_{base} \leftarrow \textsc{ConstructTripleBase}(q, \ell_1, \ell_2, \ell_3)$\;
    \tcp{Lift to dimension $n$}
    $C_{trip} \leftarrow C_{base}$\;
    \For{$k \leftarrow 3$ \KwTo $n$}{
        $C_{trip} \leftarrow \textsc{LiftCycle}(C_{trip}, k, q)$\;
    }
    Merge $C_{trip}$ into $S$ at vertex $\mathbf{0}$\;
}

\BlankLine
\tcc{Case: Remaining directions (Pairwise Construction)}
Partition $\mathcal{D}$ into pairs $\{(\ell_{2i}, \ell_{2i+1})\}$\;
\For{each pair $(\ell_a, \ell_b)$ in partition}{
    $C_{pair} \leftarrow \textsc{ConstructPairwise}(\ell_a, \ell_b, n, q)$\;
    Merge $C_{pair}$ into $S$ at vertex $\mathbf{0}$\;
}

\Return $S$\;
\end{algorithm}

\bibliographystyle{amsalpha}  % or plain, alpha, etc.
\bibliography{reference}     % <-- your references.bib
\end{document}